\UseRawInputEncoding
\documentclass[a4paper, 12pt,reqno]{amsart}
\usepackage{amsmath,amssymb,amsthm,enumerate}%,backref,cite
\usepackage{textcomp}
\allowdisplaybreaks
\theoremstyle{plain}
\newtheorem{theorem}{Theorem}
\newtheorem*{theorem*}{Theorem}

\newtheorem*{lemma*}{Lemma}
\theoremstyle{definition}

\newtheorem*{definition*}{Definition}
\theoremstyle{remark}
\newtheorem{remark}{Remark}

\newtheorem*{remark*}{Remark}

\newtheorem*{statement*}{Statement}

\begin{document}
\title[Relationships between positive and sign-variable expansions]{On one map between singular expansions}

\author{Symon Serbenyuk}

\subjclass[2010]{11K55, 11J72, 26A27, 11B34,  39B22, 39B72, 26A30, 11B34.}

\keywords{ Salem function, systems of functional equations,  complicated local structure}

\maketitle
\text{\emph{simon6@ukr.net}}\\
\text{\emph{Kharkiv National University of Internal Affairs,}}\\
\text{\emph{L.~Landau avenue, 27, Kharkiv, 61080, Ukraine}}
\begin{abstract}

The present article deals with properties of one map between two expansions of real numbers of the Salem type. Differential, integral, and other properties of the function were considered.  
\end{abstract}

\section{Introduction}

 In      \cite{Salem1943}, Salem modeled a singular function of the form
$$
S(x)=S\left(\Delta^q _{i_1i_2...i_k...}\right)=\beta_{i_1}+ \sum^{\infty} _{k=2} {\left(\beta_{i_k}\prod^{k-1} _{i=1}{p_i}\right)}:=\Delta^{P} _{i_1i_2...i_k...}\in [0,1],
$$
where $q>1$ is a fixed positive integer,
$$
\sum^{\infty} _{k=1}{\frac{i_k}{q^k}}:=\Delta^q _{i_1i_2...i_k...}=x\in [0,1]
$$
 is the $q$-ary representation of a number $x$ and $i_k \in \{0, 1, \dots , q-1\}$, as well as  $P=(p_0, p_1, \dots , p_{q-1})$ is a finite fixed sequence of numbers such that the conditions  $p_j>0$, $j=\overline{0, q-1}$,  and $p_0+p_1+\dots + p_{q-1}=1$ hold. Also,
$$
\beta_{i_k}=\begin{cases}
0&\text{if $i_{k}=0$}\\
\sum^{i_{k}-1} _{l=0} {p_l}&\text{if $i_{k}\ne 0$.}
\end{cases}
$$
 In the case when $q=2$, this function is one of the simplest examples of singular functions. Let us note that generalizations of the Salem function can be non-differentiable functions or those that do not have a derivative on a certain set.  Many researches are devoted to the Salem function and its generalizations were  (for example, see \cite{ACFS2017, Kawamura2010, Symon2015, Symon2017, Symon2019} and references in these papers).

One can note an analytical approach for modelling of the Salem function by the classical definition ofa  distribution function.  
Let $\eta$ be a random variable, that defined by the $q$-ary representation
$$
\eta= \frac{\xi_1}{q}+\frac{\xi_2}{q^2}+\frac{\xi_3}{q^3}+\dots +\frac{\xi_{k}}{q^{k}}+\dots   = \Delta^{q} _{\xi_1\xi_2...\xi_{k}...},
$$
where $\xi_k=i_k$
and digits $\xi_k$ $(k=1,2,3, \dots )$ are random and taking the values $0,1,\dots ,q-1$ with positive probabilities ${p}_{0}, {p}_{1}, \dots , {p}_{q-1}$. That is $\xi_k$ are independent and  $P\{\xi_k=i_k\}={p}_{i_k}$, $i_k \in \{0, 1, \dots , q-1\}$.

From the definition of a distribution function and the following expressions 
$$
\{\eta<x\}=\{\xi_1<\alpha_1(x)\}\cup\{\xi_1=i_1(x),\xi_2<i_2(x)\}\cup \ldots 
$$
$$
\cup\{\xi_1=i_1(x),\xi_2=i_2(x),\dots ,\xi_{k-1}=i_{k-1}(x), \xi_{k}<i_{k}(x)\}\cup \dots,
$$
$$
P\{\xi_1=i_1(x),\xi_2=i_2(x),\dots ,\xi_{k-1}=i_{k-1}(x), \xi_{k}<i_{k}(x)\}
=\beta_{i_{k}(x)}\prod^{k-1} _{j=1} {{p}_{i_{j}(x)}}.
$$
Whence the following  is true: \emph{the distribution function  ${f}_{\eta}$ of the random variable $\eta$ can be represented in the following form
$$
{f}_{\eta}(x)=\begin{cases}
0&\text{whenever $x< 0$}\\
\beta_{i_1(x)}+\sum^{\infty} _{k=2} {\left({\beta}_{i_k(x)} \prod^{k-1} _{j=1} {{p}_{i_j(x)}}\right)}&\text{whenever $0 \le x<1$}\\
1&\text{whenever $x\ge 1$,}
\end{cases}
$$
where ${p}_{i_{j(x)}}>0$.}

In addition, the function
$$ 
{f}(x)=\beta_{i_1(x)}+\sum^{\infty} _{n=2} {\left({\beta}_{i_n(x)}\prod^{n-1} _{j=1} {{p}_{i_j(x)}}\right)},
$$
can be used as a representation of numbers from $[0,1]$. That is,
$$
x=\Delta^{P} _{i_1(x)i_2(x)...i_k(x)...}=\beta_{i_1(x)}+\sum^{\infty} _{k=2} {\left({\beta}_{i_k(x)}\prod^{n-1} _{j=1} {{p}_{i_j(x)}}\right)},
$$
where $P=\{p_0,p_1,\dots , p_{s-1}\}$, $p_0+p_1+\dots+p_{s-1}=1$, and $p_i>0$ for all $i=\overline{0,s-1}$. The last-mentioned representation is \emph{the P-representation of numbers from $[0,1]$}.

As noted in the paper~\cite{Symon2023-numeral-systems}, there exists the tendency to model and to investigate  alternating expansions of real numbers under the condition that  corresponding positive expansions of real numbers are well-known. Beta- (\cite{Renyi1957}) and (- beta)-expansions (\cite{IS2009}),  positive (\cite{15}) and alternating (\cite{10}) L\"uroth series, as well as  positive and alternating Engel series~(for example, see \cite{Engel1, Engel 3} and references in \cite{Engel2, Engel4}) are such examples.  

According to the mentioned tendency, one can consider an alternating expansion which more general case was introduced in~\cite{2016Serbeyuk-KNU}:
\begin{equation}
\label{def: nega-Q}
[0,1]\ni x= \sum^{i_1-1} _{i=0}{p_{i}}+\sum^{\infty} _{k=2}{\left((-1)^{k-1}\tilde \delta_{i_k}\prod^{k-1} _{j=1}{\tilde p_{i_j}}\right)}+\sum^{\infty} _{k=1}{\left(\prod^{2k-1} _{j=1}{\tilde q_{i_j}}\right)},
\end{equation}
where
$$
\tilde \delta_{i_{k}}=\begin{cases}
1,&\text{if $k$ is  even  and $i_{k}=q-1$}\\
\sum^{q-1} _{i=q-1-i_{k}} {p_{i}},&\text{if $k$ is  even  and $i_{k}\ne q-1$}\\
0,&\text{if $k$ is an odd number  and $i_{k}=0$}\\
\sum^{i_k-1} _{i=0}{p_{i}},&\text{if $k$ is odd and $i_{k}\ne0$,}\\
\end{cases}
$$
and the first sum in  expression  \eqref{def: nega-Q} equals  $0$, if  $i_1=0$.  

An expansion of a number  $x$ by  series \eqref{def: nega-Q} is called  \emph{the nega-P-expansion of  $x$} and denotes by  $\Delta^{-P} _{i_1i_2...i_k...}$. The last-mentioned notation is called   \emph{the nega-P-representation of $x$} \cite{2016Serbeyuk-KNU}.

 Numbers from  some countable subset of $[0,1]$ have two different nega-P-representations, i.e., 
$$
\Delta^{-P} _{i_1i_2...i_{k-1}i_k[q-1]0[q-1]0[q-1]...}=\Delta^{-P} _{i_1i_2...i_{k-1}[i_k-1]0[q-1]0[q-1]...}, ~i_k\ne 0.
$$
These numbers are called  \emph{nega-P-rationals} and the rest of the numbers from $[0,1]$ are called  \emph{nega-P-irrationals}.

Let us remark that the following relationship is true:
$$
x=\Delta^{-P} _{i_1i_2...i_k...} \equiv \Delta^{P} _{i_1[q-1-i_2]i_3...i_{2k-1}[q-1-i_{2k}]...}.
$$
That is, we have
$$
\beta_{i_1}+\sum^{\infty} _{k=2} {\left({\tilde\beta}_{i_k}\prod^{k-1} _{j=1} {{\tilde p}_{i_j}}\right)},
$$
where
$$
\tilde  p_{i_k}=\begin{cases}
p_{i_k}&\text{if  $k$ is odd}\\
p_{q-1-i_k}&\text{if $k$ is even}
\end{cases}
$$
and
$$
\tilde\beta_{i_k}=\begin{cases}
\beta_{i_k}&\text{if  $k$ is odd}\\
\beta_{q-1-i_k}&\text{if $k$ is even}.
\end{cases}
$$

Investigations of related positive and alternating expansions of real numbers are useful for modelling and studying of  pathological (see \cite{Wikipedia-pathology}) mathematical objects (in $\mathbb R^1$, fractals (see \cite{Bunde1994, Falconer1997, Falconer2004, Mandelbrot1977, Mandelbrot1999, Moran1946, sets1, sets2, sets} and references therein),  singular (for example, \cite{{Salem1943}, {Zamfirescu1981}, {Minkowski}, {S.Serbenyuk 2017}}),  nowhere monotonic \cite{Symon2017, Symon2019}, and nowhere differentiable functions  (for example, see \cite{{Bush1952}, {Serbenyuk-2016}}, etc.)).  An interest in these functions can be explained by their  connection with modelling of  real objects and  processes in  economics, physics, and technology, etc.,  as well as with different areas of mathematics (for example, see~\cite{BK2000, ACFS2011, Kruppel2009, OSS1995,    Symon21, Symon21-1, Sumi2009, Takayasu1984, TAS1993, Symon2021}). 

Let  $2<q$  be a fixed positive integer, $A=\{0,1,\dots ,q-1\}$, 
$A_0=A \setminus \{0\}=\{1,2,\dots , q -1\}$,  and
$$
 L :=  (A_0)^{\infty}= (A_0) \times  (A_0) \times  (A_0)\times\dots  
$$
be the space of one-sided sequences of  elements of $ A_0$.

As an example of modifications of properties of fractals in various numeral systems, one can note the set  $\mathbb S_{(P,u)}$
$$
\mathbb S_{(P,u)}\equiv\left\{x: x= \Delta^{P}_{{\underbrace{u...u}_{i_1-1}} i_1{\underbrace{u...u}_{i_2 -1}}i_2 ...{\underbrace{u...u}_{ i_k -1}}\alpha_k...},  (\alpha_k) \in L, \alpha_k \ne u, \alpha_k \ne 0 \right\}, 
$$
where $u=\overline{0,q-1}$, the parameters $u$ and $q$ are fixed for the set $\mathbb  S_{(P,u)}$. 
and the same set $\mathbb S_{(P,u)}$ in terms of corresponding alternating expansion
$$
\mathbb S_{(-P,u)}\equiv\left\{x: x= \Delta^{-P}_{{\underbrace{u...u}_{i_1-1}} i_1{\underbrace{u...u}_{i_2 -1}}i_2 ...{\underbrace{u...u}_{ i_k -1}}\alpha_k...},  (\alpha_k) \in L, \alpha_k \ne u, \alpha_k \ne 0 \right\}.
$$
The next theorems formulate that the self-similar fractal $\mathbb S_{(P,u)}$ transforms into the non-self-similar fractal $\mathbb S_{(-P,u)}$
under the map 
$$
g: x=\Delta^{P} _{i_1i_2...i_k...} \to \Delta^{-P} _{i_1i_2...i_k...} =y=g(x).
$$

\begin{theorem}{\cite{sets}.}
\label{th: S(p,u)}
The set $\mathbb S_{(P,u)}$ is an uncountable, perfect, and nowhere dense set of zero Lebesgue measure and also is  a self-similar fractal whose Hausdorff dimension $\alpha_0 (\mathbb S_{(P,u)})$  satisfies the following equation 
$$
\sum _{i\in A_0\setminus\{u\}} {\left(p_ip^{i-1} _u\right)^{\alpha_0}}=1.
$$
\end{theorem}
\begin{theorem}{\cite{{sets2}}.}
\label{th: the second main theorem}
In the general case,  the set  $\mathbb S_{(-P,u)} $ is not a self-similar fractal,  the Hausdorff dimension $\alpha_0 (\mathbb S_{(-P,u)})$ of which can be calculated by the formula: 
$$
\alpha_0=\liminf_{k\to\infty}{\alpha_k},
$$
where $(\alpha_k)$ is a sequence of numbers satisfying the equation
$$
\left(\sum_{\substack{c_1 \text{is odd}\\ c_1\in \overline{A}}}{\left(\omega_{2,c_1}\right)^{\alpha_1} }+\sum_{\substack{c_1 \text{is even}\\ c_1\in \overline{A}}}{\left(\omega_{4,c_1}\right)^{\alpha_1} }\right)\times
$$
$$
\times\prod^k _{i=2}{\left(\sum_{\substack{c_i \text{is odd}\\ c_i\in \overline{A}}}{N_{1,c_i}\left(\omega_{1,c_i}\right)^{\alpha_i} }+\sum_{\substack{c_i \text{is odd}\\ c_i\in \overline{A}}}{N_{2,c_i}\left(\omega_{2,c_i}\right)^{\alpha_i} }+\sum_{\substack{c_i \text{is even}\\ c_i\in \overline{A}}}{N_{3,c_i}\left(\omega_{3,c_i}\right)^{\alpha_i} }+\sum_{\substack{c_i \text{is even}\\ c_i\in \overline{A}}}{N_{4,c_i}\left(\omega_{4,c_i}\right)^{\alpha_i} }\right)} =1.   
$$
Here $N_{j,c_i}$ ($j=\overline{1,4}, 1<i\in\mathbb N$) is the number of cylinders $\Delta^{(-P,u)} _{c_1c_2...c_i}$ for which
$$
\frac{d\left(\Delta^{(-P,u)} _{c_1c_2...c_{i-1}c_i}\right)}{d\left(\Delta^{(-P,u)} _{c_1c_2...c_{i-1}}\right)}=\omega_{j,c_i}.
$$
Also, 
$$
\omega_{1,c_i}=\underbrace{p_{s-1-u}p_u\ldots p_{s-1-u}p_u}_{c_i-1}p_{s-1-c_i}\frac{d\left(\overline{\mathbb S_{(P,u)}}\right)}{d\left(\underline{\mathbb S_{(P,u)}}\right)}~~~\text{for an  odd number $c_i$},
$$
$$
\omega_{2,c_i}=\underbrace{p_up_{s-1-u}\ldots p_up_{s-1-u}}_{c_i-1}p_{c_i}\frac{d\left(\underline{\mathbb S_{(P,u)}}\right)}{d\left(\overline{\mathbb S_{(P,u)}}\right)}~~~\text{for an odd number $c_i$},
$$
$$
\omega_{3,c_i}=\underbrace{p_{s-1-u}p_u\ldots p_{s-1-u}p_up_{s-1-u}}_{c_i-1}p_{c_i}~~~\text{for an even number $c_i$},
$$
$$
\omega_{4,c_i}=\underbrace{p_up_{s-1-u}\ldots p_up_{s-1-u}p_{u}}_{c_i-1}p_{s-1-c_i}~~~\text{for an even number $c_i$}.
$$
In addition, $N_{1,c_i}+N_{2,c_i}=l(m+l)^{i-1}$ and $N_{3,c_i}+N_{4,c_i}=m(m+l)^{i-1}$, where $l$ is the number of odd numbers in the set  $\overline{A}=A\setminus\{0,u\}$ and $m$ is the number of even numbers in $\overline{A}$.
\end{theorem}

More explanations and examples are given in \cite{{sets1}}.

The last sets are sets of the Moran type and  play an important role in multifractal analysis/formalism and especially the refined multifractal formalism (for example, see  \cite{2021-1, Wu2005, Wu2005(2), XW2008} and references therein). It is well-known  (see \cite{sets1} and references therein) that the multifractal analysis was proven to be a very useful technique in the analysis of measures, both in theory and applications.  Also, the multifractal and  fractal analysis allows one to perform a certain classification of singular measures.  A general form of multifractal formalism was  introduced by L. Olsen in \cite{Olsen1995}. In a general way, one can describe the following topics of related investigations: the singularity of Hewitt--Stromberg measures on Bedford--McMullen carpets \cite{AB2021-AM}, the mutual singularity of certain measures (see \cite{{Yuan2019-N}, {DS2020-ERA}, {XW2008-Fr}, {Selmi2021-BPASciM}, {DS2021-RM}} and references therein), dimensions of measures \cite{{RS2021-JGA}, {HLW2020-JMAA}, {Selmi2022-ASM}}; the discuss of applications of the singularly of continuous measures \cite{DS2020-ERA}; the Hausdorff and packing dimensions of the level sets in terms of the Legendre transform of some free energy function in analogy with the usual thermodynamic theory (\cite{{Selmi2022-ASM}, {Yuan2019-N}, {2021-1}} and references therein);  multifractal properties  of homogeneous Moran fractals associated with Fibonacci sequence \cite{Wu2005(2)}, multifractal properties \cite{{Wu2005}, {WX2011-ChSF}}; applications of  fractals of Moran's types (for example, see \cite{DSM2021-EPhJ}); the  multifractal formalism based on the Hewitt--Stromberg measures (\cite{2021-1}); the multifractal formalism based on random and non-random self-similar measures, on self-conformal measures, on self-affine and for Moran measures (see \cite{{2021-1}, {DS2021-Pr}} and references therein).

Based on the mentioned arguments on applications relationships between various numeral systems for modelling and studying in the theory of pathological mathematical objects having complicated local structure, let us investigate a map of P-representation into sign-variable expansion, which is called \emph{$\pm P$-representation} and is related with P-representation by the following relationship:
$$
\Delta^{\pm P} _{i_1i_2...i_k...}\equiv \Delta^{P} _{\bar i_1 \bar i_2...\bar i_k...},
$$
where
$$
 \bar {i_k}=\begin{cases}
{i_k}&\text{if  $k \notin N_B$}\\
{q-1-i_k}&\text{if $k \in N_B$}
\end{cases}
$$
and
 $\mathbb N_B$ is a fixed subset of positive integers.

So, let us consider the map
$$
g: x=\Delta^{P} _{i_1i_2...i_k...} \to \Delta^{\pm P} _{i_1i_2...i_k...} =y=g(x),
$$
i.e., 
$$
g(x)=\bar \beta_{i_1}+\sum^{\infty} _{k=2} {\left({\bar\beta}_{i_k}\prod^{k-1} _{j=1} {{\bar p}_{i_j}}\right)},
$$
where
$$
\bar p_{i_k}=\begin{cases}
p_{i_k}&\text{if  $k \notin N_B$}\\
p_{q-1-i_k}&\text{if $k\in N_B$}
\end{cases}
$$
and
$$
\bar\beta_{i_k}=\begin{cases}
\beta_{i_k}&\text{if  $k \notin N_B$}\\
\beta_{q-1-i_k}&\text{if $k \in N_B$}.
\end{cases}
$$

%%%%%%%%%%%%%%%%%%%%%%%%%%%%%%%
\section{Auxiliary properties}
%%%%%%%%%%%%%%%%%%%%%%%%%%%%%%%

In this section, let us describe known auxiliary properties which are useful for proofs of the main statements. 

{\itshape A cylinder of rank $m$ with  base $c_1c_2\ldots c_m$} is a set $\Lambda^{P} _{c_1c_2\ldots c_m}$ formed by all numbers
of the segment  $[0,1]$ with P-representations in which the first $m$ digits coincide with $c_1,c_2,\dots ,c_m$, respectively, i.e.,
$$
\Lambda^{P} _{c_1c_2\ldots c_m} :=\left\{x: x=\Delta^{P} _{i_1i_2\ldots i_k\ldots}, i_j=c_j, j=\overline{1,m}\right\}.
$$

Cylinders $\Lambda^{P} _{c_1c_2\ldots c_m}$ have the following properties: 

\begin{enumerate}

\item any cylinder $\Lambda^{P} _{c_1c_2\ldots c_m}$ is a closed interval;
\item 
$$
\inf \Lambda^{P} _{c_1c_2\ldots c_m}= \Delta^{P} _{c_1c_2\ldots c_m000...},
\sup \Lambda^{P} _{c_1c_2\ldots c_m}= \Delta^{P} _{c_1c_2\ldots c_m[q-1][q-1][q-1]...};
$$
\item
$$
| \Lambda^{P} _{c_1c_2\ldots c_m}|=p_{c_1}p_{c_2}\cdots p_{c_m}; 
$$
\item
$$
 \Lambda^{P} _{c_1c_2\ldots c_mc}\subset  \Lambda^{P} _{c_1c_2\ldots c_m};
$$
\item
$$
 \Lambda^{P} _{c_1c_2\ldots c_m}=\bigcup^{q-1} _{c=0} { \Lambda^{P} _{c_1c_2\ldots c_mc}};
$$
\item
$$
\lim_{m \to \infty} { |\Lambda^{P} _{c_1c_2\ldots c_m}|}=0;
$$
\item
$$
\frac{| \Lambda^{P} _{c_1c_2\ldots c_mc_{m+1}}|}{| \Lambda^{P} _{c_1c_2\ldots c_m}|}=p_{c_{m+1}};
$$
\item 
$$
\sup\Lambda^{P} _{c_1c_2...c_mc}=\inf  \Lambda^{P} _{c_1c_2...c_m[c+1]},
$$
where $c \ne q-1$;
\item
$$
\bigcap^{\infty} _{m=1} {\Lambda^{P} _{c_1c_2\ldots c_m}}=x=\Delta^{P} _{c_1c_2\ldots c_m\ldots}.
$$
\end{enumerate}

A number $x \in[0,1]$ is called   {\itshape P-rational} if 
$$
x=\Delta^{P} _{i_1i_2\ldots i_{k-1}i_k 000\ldots}
$$
or
$$
x=\Delta^{P} _{i_1i_2\ldots i_{k-1}[i_k-1][q-1][q-1][q-1]\ldots}.
$$
The  other  numbers in $[0,1]$ are called {\itshape P-irrational}.

The mapping is defined by
$$
\sigma(x)=\sigma \left(\Delta^{P} _{i_1i_2...i_k...}\right)=\beta_{i_2}+\sum^{\infty} _{k=3}{\left(\beta_{i_k}\prod^{k-1} _{j=2}{p_{i_j}}\right)}
$$
is called \emph{the shift operator $\sigma$ of the $P$-expansion of  $x$}. 

The following mapping
$$
\sigma^n(x)=\beta_{i_{n+1}}+\sum^{\infty} _{k=n+2}{\left(\beta_{i_k}\prod^{k-1} _{j=n+1}{p_{i_j}}\right)}
$$
is called \emph{the shift operator $\sigma^n$ of rank $n$ of the $P$-expansion of a number $x$}.

Since  $x=\beta_{i_1}+p_{i_1}\sigma(x)$, 
$$
\sigma(x)=\frac{x-\beta_{i_1}}{p_{i_1}}.
$$

It is easy to see that the following expressions are true. 
$$
\sigma^n(x)=\frac{1}{p^n _{0}}\Delta^{P} _{\underbrace{0...0}_{n}i_{n+1}i_{n+2}...},
$$
\begin{equation}
\label{shift operator}
\sigma^{n-1}(x)=\beta_{i_n}+p_{i_n}\sigma^n(x)=\frac{1}{p^{n-1} _{0}}\Delta^{P} _{\underbrace{0...0}_{n-1}i_{n}i_{n+1}...}.
\end{equation}

Such operator is useful to formulate conditions of representations of rational numbers  (see~\cite{{S2023-Communications-in-Mathematics}, {Serbenyuk20Tatra}, {Symon2021-ActaMathHungarica}} Com Math, Acta, ) and to model   functions with complicated local structure (for example, see \cite{{Symon2015}, {Symon2017}, {S. Serbenyuk systemy rivnyan 2-2}, {Symon2019}}, etc.).

%%%%%%%%%%%%%%%%%%%%%%%%%%%%%%%%%%%%%%%%%%%%%%%%%
\section{The main object: the monotonicity and differential properties}
%%%%%%%%%%%%%%%%%%%%%%%%%%%%%%%%%%%%%%%%%%%%%%%%%

Suppose $q>1$ is a fixed positive integer,  as well as $P :=(p_0, p_1, \dots, p_{q-1})$ is a fixed set of real numbers for which the conditions $p_j>0$ for all $j=\overline{0,q-1}$ and $p_0+p_1\dots + p_{q-1}=1$ hold. Suppose $\mathbb N_B$ is a fixed subset of positive integers.

Let us consider the function
$$
g(x)=\bar \beta_{i_1}+\sum^{\infty} _{k=2} {\left({\bar\beta}_{i_k}\prod^{k-1} _{j=1} {{\bar p}_{i_j}}\right)},
$$
where
$$
\bar p_{i_k}=\begin{cases}
p_{i_k}&\text{if  $k \notin N_B$}\\
p_{q-1-i_k}&\text{if $k\in N_B$}
\end{cases}
$$
and
$$
\bar\beta_{i_k}=\begin{cases}
\beta_{i_k}&\text{if  $k \notin N_B$}\\
\beta_{q-1-i_k}&\text{if $k \in N_B$}.
\end{cases}
$$

As an auxiliary notation, we use the following relationship:
$$ 
y=g(x)= \Delta^{\pm P} _{i_1i_2...i_k...} \equiv \Delta^{P} _{\bar i_1 \bar i_2... \bar i_k...},
 $$
where
$$
 \bar {i_k}=\begin{cases}
{i_k}&\text{if  $k \notin N_B$}\\
{q-1-i_k}&\text{if $k \in N_B$}.
\end{cases}
$$

In addition, we have $\Delta^P _{[q-1][q-1][q-1]...}=1$, $\Delta^P _{000...}=0$, and $0\le \sigma^m (x)  \le 1$ for any $x\in [0, 1]$ and positive integers $m$.

\begin{theorem} If the condition $N_B\in\{\emptyset, \mathbb N\}$ holds, then the function $g$ is  continuous on $[0, 1]$.

In  other cases, the  function   $g$ is continuous at $P$-irrational  points, and the $P$-rational  points are points of discontinuity (jumps) of the function. Also, the set of all points of discontinuity of  $g$ can be a finite or  enumerable set. This fact depends on the number of elements of $N_B$ (finite or infinite).
\end{theorem}
\begin{proof} Let use the standard technique considering two cases: the case of a $P$-rational point, as well as  of a $P$-irrational point.

Suppose  $x_0, x \in \Lambda^{P} _{c_1c_2...c_m}$  are arbitrary $P$-irrational points such that $x\ne x_0$,  $x=\Delta^{P} _{i_1i_2...i_k ..}$  and $x_0=\Delta^{P} _{j_1j_2...j_k ...}$; then there exists a positive integer $m$  such that  the condition $i_{r}=j_{r}$ holds  for all $r=\overline{1,m}$ and $i_{m+1}\ne j_{m+1}$.  

Since $g$ is a bounded function, $0\le f(x) \le 1$, we obtain 
\begin{equation*}
\begin{split}
g(x)-g(x_0)&=\bar \beta_{i_1}+\sum^{m} _{k=2}{\left(\bar \beta_{i_k}\prod^{k-1} _{t=1}{\bar p_{i_t}}\right)}+\prod^{m} _{u=1}{\bar p_{i_u}}\left(\bar \beta_{i_{m+1}}+\sum^{\infty} _{s=m+2}{\left(\bar \beta_{i_s}\prod^{s-1} _{l=m+1}{\bar p_{i_l}}\right)}\right)\\
&-\bar \beta_{i_1}-\sum^{m} _{k=2}{\left(\bar \beta_{i_k}\prod^{k-1} _{t=1}{\bar p_{i_t}}\right)}-\prod^{m} _{u=1}{\bar p_{i_u}}\left(\bar \beta_{j_{m+1}}+\sum^{\infty} _{s=m+2}{\left(\bar \beta_{j_s}\prod^{s-1} _{l=m+1}{\bar p_{j_l}}\right)}\right)\\
&=\prod^{m} _{u=1}{\bar p_{i_u}}\left(g(\sigma^m(x))- g(\sigma^m(x_0))\right)\le (1-0)\prod^{m} _{u=1}{\bar p_{i_u}}=\prod^{m} _{u=1}{\bar p_{i_u}}.
\end{split}
\end{equation*}
Hence
$$
\lim_{m\to \infty}{|g(x)-g(x_0)|}=\lim_{m\to \infty}{\prod^{m} _{u=1}{\bar p_{ i_u}}}\le\lim_{m\to \infty}{\left(\max\{p_0, p_1, \dots, p_{q-1}\}\right)^{m}}=0.
$$

So, $\lim_{x\to x_0}{g(x)}=g(x_0)$, i.e., our function  is continuous at any $P$-irrational point. 

Let $x_0$ be a $P$-rational number, i.e.,
$$
x_0=x_1=\Delta^{P} _{i_1i_2...i_{m-1}i_m000...}=\Delta^{P} _{i_1i_2...i_{m-1}[i_m-1][q-1][q-1][q-1]...}=x_2,~~~i_m \ne 0.
$$
Then 
\begin{equation*}
\begin{split}
\Delta  g= g(x_1)-g(x_2)&=\prod^{m-1} _{j=1}{\bar p_{i_j}}\left(\bar \beta_{i_m}-\bar \beta_{i_m-1}+\bar p_{i_m}g(\sigma^m(x_1))-\bar p_{i_m-1}g(\sigma^m(x_2))\right)\\
&=\prod^{m-1} _{j=1}{\bar p_{i_j}}\left(\bar \beta_{i_m}-\bar \beta_{i_m-1}+\bar p_{i_m}g(\sigma^m(0))-\bar p_{i_m-1}g(\sigma^m(1))\right).
\end{split}
\end{equation*}
Whence,
\begin{equation*}
\begin{split}
\frac{\Delta g}{\prod^{m-1} _{j=1}{\bar p_{i_j}}}&=\begin{cases}
p_{i_m-1}+p_{i_m}g(\sigma^m(0))-p_{i_m-1}g(\sigma^m(1))&\text{if $m\notin N_B$}\\
-p_{q-1-i_m}+p_{q-1-i_m}g(\sigma^m(0))-p_{q-i_m}g(\sigma^m(1)) &\text{if $m\in N_B$,}
\end{cases}\\
&=\begin{cases}
(1-g(\sigma^m(1)))p_{i_m-1}+p_{i_m}g(\sigma^m(0)) &\text{if $m\notin N_B$}\\
(g(\sigma^m(0))-1)p_{q-1-i_m}-p_{q-i_m}g(\sigma^m(1)) &\text{if $m\in N_B$.}
\end{cases}
\end{split}
\end{equation*}

The equality to zero holds whenever the following systems of conditions is true:
$$
\begin{cases}
1+\frac{p_{i_m}}{p_{i_m-1}}g(\sigma^m(0))= g(\sigma^m(1)) &\text{if $m\notin N_B$}\\
g(\sigma^m(0))=1+\frac{p_{q-i_m}}{p_{q-1-i_m}}g(\sigma^m(1)) &\text{if $m\in N_B$.}
\end{cases}
$$
Hence  $g$ is continuous at a $P$-rational point whenever the system of conditions
$$
\begin{cases}
g(\sigma^m(0))=0,  g(\sigma^m(1))=1 &\text{if $m\notin N_B$}\\
g(\sigma^m(0))=1, g(\sigma^m(1))=0 &\text{if $m\in N_B$}
\end{cases}
$$
holds.

Finally, 
$$
\lim_{x \to x_0 -0} {g(x)}=\Delta^{P} _{\bar i_1 \bar i_2 ... \bar i_{m-1} [\overline{i_m-1}][\overline{q-1}] [\overline{q-1}] [\overline{q-1}] ...},
$$
$$
\lim_{x \to x_0 +0} {g(x)}=\Delta^{P} _{\bar i_1 \bar i_2 ... \bar i_{m-1} \bar i_m \bar 0 \bar 0 \bar 0 ...}.
$$

Since $0\le \sigma^m (x)  \le 1$, $p_r>0$ holds for all $r=\overline{0, q-1}$, we obtain that the set of all points of discontinuity of the function $g$ is empty whenever $N_B\in\{\emptyset, \mathbb N\}$. In other cases, the set of all points of discontinuity of  $g$ can be a finite or  enumerable set. This fact depends on the number of elements of $N_B$ (finite or infinite).
\end{proof}

\begin{remark}
To reach that the function $g$ be upper semicontinuous (a right-continuous function) on the set of $P$-rational numbers from $[0, 1]$, we will not consider the $P$-representations, which have the period $(q-1)$ (without the number~$1$). That is, our function will be well-defined. 
\end{remark}

\begin{theorem} If the condition $N_B\in\{\emptyset, \mathbb N\}$ holds, then the function $g$ is  a monotonic function on the closed interval $[0, 1]$.

In  other cases, the  function   $g$ is not a monotonic function on the domain, but there exist intervals of the  monotonicity whenever the set  $N_B$ is finite, as well as there are not   intervals of the  monotonicity whenever the set  $N_B$ is infinite.
\end{theorem}
\begin{proof}  Let us prove the statement by the definition of the monotonicity. Suppose $x_1<x_2$ with  $x_1=\Delta^{P} _{i_{1}i_{2}...i_{k}...}$ and $x_2=\Delta^{P} _{j_{1}j_{2}...j_{k}...}$; then there exists a number$m$ such that $i_r=j_r$ holds for all $r=\overline{1,m}$ and $i_{m+1}<j_{m+1}$. Whence $x_1, x_2 \in \Lambda^{P} _{c_1c_2...c_m}$. 

Let us consider the images $g(x_1)$ and $g(x_2)$. Really, we have $\bar i_r=\bar j_r=\bar c_r$ holds for all $r=\overline{1,m}$,  but
$$
\bar i_{m+1}=i_{m+1}<j_{m+1}=\bar j_{m+1}~~~\text{whenever}~~~m\notin N_B
$$
and
$$
\bar i_{m+1}=q-1-i_{m+1}>q-1-j_{m+1}=\bar j_{m+1}~~~\text{whenever}~~~m\in N_B.
$$
are true.

When $m\to\infty$, we will obtain infinitely many images $g(x_1)<g(x_2)$ for $m\notin N_B$ and 
$g(x_1)>g(x_2)$ for $m\in N_B$. 
\end{proof}

\begin{theorem}
If $N_B\ne \emptyset$, then the function $g$  is a singular function almost everywhere on~$[0, 1]$.
\end{theorem}
\begin{proof}
Let us consider a closed interval $\Lambda^{P} _{c_1c_2...c_m}$, where
$$
\Lambda^{P} _{c_1c_2...c_m}=\left[\inf\Lambda^{P} _{c_1c_2...c_m}; \sup\Lambda^{P} _{c_1c_2...c_m}\right]= \left[\Delta^{P} _{c_1c_2...c_m000...}; \Delta^{P} _{c_1c_2...c_m[q-1][q-1][q-1]...}\right], 
$$
and
$$
\Delta^{P} _{c_1c_2...c_m[q-1][q-1][q-1]...}=\begin{cases}
 \Delta^{P} _{c_1c_2...c_{m-1}[c_m+1]000...} &\text{if $c_m\ne q-1$}\\
 1 &\text{in other cases.}
\end{cases}
$$
Then
$$
\sup\Lambda^{P} _{c_1c_2...c_m}- \inf\Lambda^{P} _{c_1c_2...c_m}=\prod^{m} _{u=1}{ p_{ i_u}},
$$
as well as
\begin{equation*}
\begin{split}
\mu_g\left(\Lambda^{P} _{c_1c_2...c_m}\right)&=g\left(\Delta^{P} _{c_1c_2...c_m[q-1][q-1]...}\right)-g\left(\Delta^{P} _{c_1c_2...c_m00...}\right)\\
&=\prod^{m} _{u=1}{\bar p_{ i_u}}(g(\sigma^m(\Delta^{P} _{c_1c_2...c_m[q-1][q-1]...}))-g(\sigma^m(\Delta^{P} _{c_1c_2...c_m00...})))\\
&=\prod^{m} _{u=1}{\bar p_{c_u}}(g(\sigma^m(1))-g(\sigma^m(0)))
\end{split}
\end{equation*}
and
$$
g\left( \Delta^{P} _{c_1c_2...c_{m-1}[c_m+1]000...}\right)-g\left(\Delta^{P} _{c_1c_2...c_m000...}\right)=\prod^{m-1} _{u=1}{\bar p_{ c_u}} (\bar \beta_{c_m+1}-\bar \beta_{c_m}+g(\sigma^m(0)(\bar p_{c_m+1}-\bar p_{c_m})))
$$
or
$$
g\left( \Delta^{P} _{c_1c_2...c_{m-1}[c_m+1]000...}\right)-g\left(\Delta^{P} _{c_1c_2...c_m000...}\right)=\prod^{m-1} _{u=1}{\bar p_{ c_u}} (\bar \beta_{q-1}-\bar \beta_{c_m}+(\bar p_{q-1}g(\sigma^m(1)-\bar p_{c_m}g(\sigma^m(0))))
$$

So,
$$
g^{'}(x_0)=\lim_{m\to\infty}{\frac{\mu_g\left(\Lambda^{P} _{c_1c_2...c_m}\right)}{\left|\Lambda^{P_q} _{c_1c_2...c_m}\right|}}=\lim_{m\to\infty} \prod^{m-1} _{t=1}{y_m\frac{\bar p_{c_t}}{p_{c_t}}}=\lim_{m\to\infty} \prod^{m-1} _{t=1}{y_m\frac{\bar p_{c_t}}{p_{c_t}}},
$$
where $(y_m)$ is a certain bounded sequence. Using arguments from~\cite{{Symon2023}, {Symon2024}}, we get 
$$
\lim_{m\to\infty} \prod_{{\substack{t=\overline{1, m-1},\\ t\in N_B}}}{\frac{\bar p_{c_t}}{p_{c_t}}}=0
$$
and
$$
g^{'}(x_0)=0.
$$
\end{proof}

%%%%%%%%%%%%%%%%%%%%%%%%%%%%%%%%%%%%%%%%%%%%%%%%%
\section{The main object: about the graph}
%%%%%%%%%%%%%%%%%%%%%%%%%%%%%%%%%%%%%%%%%%%%%%%%%

Let us consider auxiliary property, which is useful for the calculation of a value of the Lebesgue integral. 
\begin{theorem} 
Let $P_q=(p_0, p_1, \dots , p_{q-1})$ be a fixed tuple of real numbers such that $p_t\in (0,1)$, where $t=\overline{0,q-1}$, $\sum_t {p_t}=1$, and $0=\beta_0<\beta_t=\sum^{t-1} _{l=0}{p_l}<1$ for all $t\ne 0$. Then the following system of functional equations
\begin{equation}
\label{eq: system}
f\left(\sigma^{n-1}(x)\right)=\bar \beta_{i_{n}}+\bar p_{i_{n}}f\left(\sigma^n(x)\right),
\end{equation}
where $x=\Delta^{P} _{i_1i_2...i_k...}$, $n=1,2, \dots$, $\sigma$ is the shift operator, and $\sigma_0(x)=x$, has the unique solution
$$
f(x)=\bar \beta_{{i_1}}+\sum^{\infty} _{k=2}{\left(\bar \beta_{i_{k}}\prod^{k-1} _{r=1}{\bar p_{i_{r}}}\right)}
$$
in the class of determined and bounded on $[0, 1]$ functions. 
\end{theorem}
\begin{proof} Our statement is true, because the following properties and relationships hold: the function $g$ is a determined and bounded  function on $[0,1]$, as well as according to system~\eqref{eq: system},  we have
$$
g(x)=\bar \beta_{{i_1}}+\bar p_{{i_1}}g(\sigma(x))=\bar \beta_{{i_1}}+\bar p_{{i_1}}(\bar \beta_{{i_2}}+\bar p_{{i_2}}g(\sigma^2(x)))=\dots
$$
$$
\dots =\bar \beta_{{i_1}}+\bar \beta_{{i_2}}\bar p_{{i_1}}+\bar \beta_{{i_3}}\bar p_{{i_1}}\bar p_{{i_2}}+\dots +\bar \beta_{i_n}\prod^{n-1} _{l=1}{\bar p_{{i_l}}}+\left(\prod^{n} _{r=1}{\bar p_{{i_r}}}\right)g(\sigma^n(x)),
$$
$$
\prod^{n} _{r=1}{\bar p_{i_r}}\le \left( \max\{p_0, p_1, \dots , p_{q-1}\}\right)^n\to 0, ~~~ n\to \infty,
$$
and
$$
\lim_{n\to\infty}{g(\sigma^n(x))\prod^{n} _{r=1}{\bar p_{{i_r}}}}=0.
$$
\end{proof}

\begin{theorem}
Suppose
$$
\psi_{i_k}: \left\{
\begin{array}{rcl}
x^{'}&=&p_{i_k}x+\beta_{i_k}\\
y^{'} & = &\bar p_{i_k}y+\bar \beta_{i_k}\\
\end{array}
\right.
$$
are affine transformations for $i_k \in \{0, 1, \dots , q-1\}$ and $p_0, p_1, \dots , p_{q-1} \in (0, 1)$. Then the graph $\Gamma$ of  $g$ is the following set of $\mathbb R^2$
$$
\Gamma_{g}=\bigcup_{x \in [0,1]}{(\ldots \circ \psi_{i_{k}}\circ \ldots \circ \psi_{i_{2}} \circ \psi_{i_{1}}(x))}.
$$
\end{theorem}
\begin{proof} 
 Since $g$ is the continuous unique solution of the system~\eqref{eq: system}, it is clear that 
$$
\psi_{i_1}:
\left\{
\begin{aligned}
x^{'}&= p_{i_1}x+ \beta_{i_1}\\
y^{'}& = \bar{\beta}_{i_1}+\bar {p}_{i_1}y,\\
\end{aligned}
\right.
$$
$$
\psi_{i_2}:
\left\{
\begin{aligned}
x^{'}&=p_{i_2}x+ \beta_{i_2}\\
y^{'}& = \bar {\beta}_{i_2}+\bar {p}_{i_2}y,\\
\end{aligned}
\right.
$$
etc.
Therefore,
$$
\bigcup_{x \in [0;1]}{( ... \circ \psi_{i_k}\circ \ldots \circ \psi_{i_2} \circ \psi_{i_1}(x))}\equiv G \subset  \Gamma_{g},
$$
because
$$
g(x^{'})=g( \beta_{i_k}+ p_{i_k}x)=\bar \beta_{i_k}+\bar p_{i_k}y=y^{'}.
$$

Let $T(x_0,g(x_0))\in \Gamma_{g}$, $x_0=\Delta^{P} _{i_1i_2...i_k...}$ be  some fixed point from  $[0, 1]$. Let $x_k \in [0,1]$ such that $x_k=\sigma^k (x_0)$.

Since   the system~\eqref{eq: system} is true and the condition 
$\overline{T}\left(\sigma^{k}(x_0);g\left(\sigma^{k}(x_0)\right)\right)\in~\Gamma_{g}$ holds, it follows that 
$$
\psi_{i_k}\circ \ldots\circ \psi_{i_2} \circ \psi_{i_1}\left(\overline{T}\right)=T_0(x_0; g(x_0))\in \Gamma_{g}, ~~~i_k\in \{0, 1, \dots, q-1\},~~~k\to \infty.
$$

Whence, $\Gamma_{g}\subset G$. So,
$$
\Gamma_{g}=\bigcup_{x \in [0,1]}{(\ldots \circ \psi_{i_k}\circ \ldots \circ \psi_{i_2} \circ \psi_{i_1}(x))}.
$$
\end{proof}

Let us consider fractal properties of the graph $\Gamma_g$ of the function $g$.

\begin{theorem}
The Hausdorff dimension of the graph $\Gamma_g$ of the function $g$ is equal to $1$.
\end{theorem}
\begin{proof}
Using the definition of the  fractal cell entropy dimension  and the definition of $g$, one can see that $\Gamma_g$  belongs to $q$ rectangles (with sides $p_{i_1}$ and $\bar p_{i_1}$) from  $q^2$
first-rank rectangles:
$$
I_{[i_1, \bar i_1]}=\left[\beta_{i_1}, \beta_{i_1+1}\right]\times\left[\min\{\bar \beta_{i_1}, \bar \beta_{i_1+1}\};\max\{\bar \beta_{i_1}, \bar \beta_{i_1+1}\} \right],
$$
where $0\le \beta_{i_1}, \beta_{i_1+1}, \bar \beta_{i_1}, \bar \beta_{i_1+1}\le \beta_q=1$ and $i_1\in\{0, 1, \dots, q-1\}$.

The graph of  the function $g$ belongs to $q^2$ rectangles (with sides $p_{i_1}p_{i_2}$ and $\bar p_{i_1}\bar p_{i_2}$) from $q^4$ rectangles of rank $2$:
$$
I_{[i_1, \bar i_1][i_2,\bar i_2]}=\left[\beta_{i_1}+\beta_{i_2}p_{i_1}, \beta_{i_1}+\beta_{i_2+1}p_{i_1}\right] \times\left[\min R_2; \max R_2 \right], 
$$
where $R_2= \{\bar \beta_{i_1}+\bar \beta_{i_2}\bar p_{i_1}, \bar \beta_{\i_1}+\bar \beta_{i_2+1}\bar p_{i_1}\}$.
In addition, one can note the following:
\begin{itemize}
\item the part of the graph, which is in the rectangle $I_{[0, \bar 0]}$, belongs to $q$ rectangles  
$$
I_{[0,\bar 0][0, \bar 0]}, I_{[0, \bar 0][(1, \bar 1]}, \dots , I_{[0, \bar 0][q-1, \overline{q-1}]};
$$
\item  the part of the graph, which is in the rectangle $I_{[1, \bar 1]}$, belongs to $q$ rectangles  
$$
I_{[1,\bar 1][0, \bar 0]}, I_{[1, \bar 1][(1, \bar 1]}, \dots , I_{[1, \bar 1][q-1, \overline{q-1}]};
$$
\item $\dots \dots \dots \dots \dots \dots \dots \dots \dots \dots \dots \dots \dots \dots \dots \dots \dots \dots \dots \dots \dots$
\item  the part of the graph, which is in the rectangle $I_{[q-1, \overline{q-1}]}$, belongs to $q$ rectangles  
$$
I_{[q-1,\overline{q-1}][0, \bar 0]}, I_{[q-1, \overline{q-1}][(1, \bar 1]}, \dots , I_{[q-1, \overline{q-1}][q-1, \overline{q-1}]}.
$$
\end{itemize}

In the $r$th step, the graph of  the function $g$ belongs to $q^r$ rectangles (with sides $\prod^{r} _{t=1}{p_{i_t}}$ and $\prod^{r} _{t=1}{\bar p_{i_t}}$) from $q^{2r}$ rectangles of rank $r$. 

Whence,
$$
\widehat{H}_{\alpha} (\Gamma_g)=\lim_{\overline{r \to \infty}}{\sum_{j=\overline{1, r},~~ p_{i_j}\in  P}\left(\sqrt{\left(\prod^{r} _{j=1}{p_{i_j}}\right)^2+\left(\prod^{r} _{j=1}{\bar p_{i_j}}\right)^2}\right)^{\alpha^K (\Gamma_f)}},
$$
where $\alpha^K (E)$ is the fractal cell entropy dimension of the set $E$ (see \cite{Serbenyuk-2016} and references therein).

Suppose $a=\min \{p_0, p_1, \dots , p_{q-1}\}$ and $b=\max \{p_0,  p_1, \dots , p_{q-1}\}$. The value of $\widehat{H}_{\alpha} (\Gamma_g)$ is between values of $\widehat{H}_{\alpha} (\Gamma_g)$ for the case of \emph{squares} of rank $r$ with sides $a^r$ and $b^r$ correspondently. Hence using the number of such squares and their diameters, let us evaluate the following values. Really, 
$$
\widehat{H}_{\alpha, a} (\Gamma_g)=\lim_{\overline{r \to \infty}}{q^r\left(\sqrt{(a^r)^2+(a^r)^2}\right)^{\alpha}}=\lim_{\overline{r \to \infty}}{q^r\left(\sqrt{2(a^r)^2}\right)^{\alpha}}=\lim_{\overline{r \to \infty}}\left(2^{\frac{\alpha}{2}}\left(qa^{\alpha}\right)^{r}\right) 
$$
and
$$
\widehat{H}_{\alpha, b} (\Gamma_g)=\lim_{\overline{r \to \infty}}{q^r\left(\sqrt{(b^r)^2+(b^r)^2}\right)^{\alpha}}=\lim_{\overline{r \to \infty}}{q^r\left(\sqrt{2(b^r)^2}\right)^{\alpha}}=\lim_{\overline{r \to \infty}}\left(2^{\frac{\alpha}{2}}\left(qb^{\alpha}\right)^{r}\right).
$$
Since the graph of our function has self-similar properties,  $\left(qa^{\alpha}\right)^{r}\to 0$ and$\left(qb^{\alpha}\right)^{r}\to 0$ for $r\to \infty$ and for large $\alpha>1$, we obtain $1=\alpha_2 \le \alpha_0(\Gamma_g)\le \alpha_1=1$, where $\alpha_0$ is the Hausdorff dimension. 
\end{proof}

%%%%%%%%%%%%%%%%%%%%%%%%%%%%%%%%
\section{Integral properties}
%%%%%%%%%%%%%%%%%%%%%%%%%%%%%%%%
\begin{theorem}
For the Lebesgue integral, the following equality holds:
$$
\int^1 _0 {g(x)dx}=\frac{\sum^{q-1}_{t=0}{\bar \beta_{t} p_t}}{1-\sum^{q-1} _{t=0}{\bar p_{t}p_t}}.
$$
\end{theorem}
=======================================

\begin{proof}
 Let us begin with some equalities which are useful for the future calculations:
$$
x=\beta_{i_1}+p_{i_1}\sigma(x)
$$
and
$$
dx=p_{i_1}d(\sigma(x)),
$$
as well as
$$
d(\sigma^{n-1} (x))=p_{i_n}d(\sigma^{n} (x))
$$
for all $n=1, 2, 3, \dots$

Let us calculate the Lebesgue integral
$$
I:=\int^1 _0{g(x)dx}=\sum^{q-1} _{t=0}{\int^{\beta_{t+1}} _{\beta_t}{g(x)dx}}=\sum^{q-1} _{t=0}{\int^{\beta_{t+1}} _{\beta_t}{(\bar \beta_{t}+\bar p_{t}g(\sigma(x))})dx}
$$
$$
=\bar \beta_{0}p_0+\bar \beta_{1}p_1+\dots + \bar \beta_{q-1}p_{q-1}+\sum^{q-1} _{t=0}{\int^{\beta_{t+1}} _{\beta_t}{\bar p_{t}g(\sigma(x))dx}}
$$
$$
=\sum^{q-1}_{t=0}{\bar \beta_{t}p_t}+\sum^{q-1} _{t=0}{\bar p_{t}\int^{\beta_{t+1}} _{\beta_t}{p_tf(\sigma(x))d(\sigma(x))}}
$$
$$
=\sum^{q-1}_{t=0}{\bar \beta_{t}p_t}+\sum^{q-1} _{t=0}{\bar p_{t}p_t\int^{1} _{0}{g(\sigma(x))d(\sigma(x))}}.
$$

Let us denote
$$
U:= \sum^{q-1}_{t=0}{\bar \beta_{t}p_t},
$$
$$
W:= \sum^{q-1} _{t=0}{\bar p_{t}p_t},
$$
and
$$
I_n:=\int^{1} _{0}{g(\sigma^n(x))d(\sigma^n(x))}.
$$
Then we obtain
$$
I= U+WI=U+W(U+WI^2)=U+UW+W^2(U+WI^3)=\dots =\frac{U}{1-W}.
$$

So,
$$
I=\frac{\sum^{q-1}_{t=0}{\beta_{\theta(t)}p_t}}{1-\sum^{q-1} _{t=0}{p_{\theta(t)}p_t}}.
$$
\end{proof}

\section*{Statements and Declarations}
\begin{center}
{\bf{Competing Interests}}

\emph{The author states that there is no conflict of interest.}
\end{center}

\begin{center}
{\bf{Data Availability Statement}}

\emph{There are not suitable for this research.}
\end{center}


\begin{thebibliography}{9}

\bibitem{AB2021-AM}
Najmeddine Attia and Bilel Selmi, On the mutual singularity of Hewitt-Stromberg measures, \emph{Analysis Mathematica}  47 (2021), 273-283.

\bibitem{2021-1}
N. Attia and  B. Selmi,  A Multifractal Formalism for Hewitt–Stromberg Measures, \emph{ J. Geom. Anal.} {\bf{31}} (2021), 825--862. 



\bibitem{ACFS2011}
E.~de Amo,  M.~D\'iaz~Carrillo,  and  J.~Fern\'andez-S\'anchez, On duality of aggregation
operators and k-negations, \emph{Fuzzy Sets and Systems}, \textbf{181} (2011), 14--27.

\bibitem{ACFS2017}
E.~de~Amo, M.~D\'iaz~Carrillo,  and  J.~Fern\'andez-S\'anchez, A Salem generalized function, \emph{Acta Math. Hungar.}, \textbf{151} (2017), no.~2,  361--378. https://doi.org/10.1007/s10474-017-0690-x



\bibitem{BK2000}
L.~Berg and M.~Kruppel, De Rham's singular function and related
functions, \emph{Z. Anal. Anwendungen.}, \textbf{19} (2000), no.~1,  227--237.


\bibitem{Bush1952}  K.~A.~Bush, Continuous functions without derivatives,  \emph{Amer. Math. Monthly}, \textbf{59}   (1952),  222--225. 

\bibitem{Bunde1994}  A.~Bunde and  S.~Havlin, \textsl{Fractals in Science}, Springer-Verlag  (Berlin, 1994).

\bibitem{C1869}{ G.~Cantor},
 {Ueber die einfachen Zahlensysteme,}
 {\itshape Z. Math. Phys.},  {\bf 14} (1869), 121--128. (German) 

\bibitem{DS2020-ERA}
Zied Douzi and Bibel Selmi, On the mutual singularity of multifractal measures, \emph{Electron. Res. Arch.} 28 (2020), 423-432.

\bibitem{DS2021-Pr}
Z. Douzi et al. Another example of the mutual singularity of multifractal measures, \emph{Proyecciones}
40 (2021), 17-33.

\bibitem{DS2021-RM}
Zied Douzi and Bibel Selmi, On the mutual singularity of Hewitt-Stromberg measures for which the
multifractal functions do not necessarily coincide, \emph{Ricerche di Matematica}, https://doi.org/10.1007/s11587-021-00572-6

\bibitem{DSM2021-EPhJ}
Zied Douzi, Bilel Selmi,  and Anouar Ben Mabrouk,  The refined multifractal formalism of some homogeneous Moran measures, \emph{Eur. Phys. J. Spec. Top.}  https://doi.org/10.1140/epjs/s11734-021-00318-3


\bibitem{Engel1}
F. Engel, Entwicklung der Zahlen nach Stammbr\"uchen,\emph{ Verhandlungen der 52. Versammlung deutscher Philologen und Schulm\"anner in Marburg}, 1913, pp. 190-191. 

\bibitem{Engel2}
P.~Erd\"os and J. O.~Shallit, New bounds on the length of finite pierce and Engel series, 
 \emph{Journal de Th\'eorie des Nombres de Bordeaux},  3 (1991), no. 1, pp. 43--53. 

\bibitem{Engel4}
Lulu Fang, Large and moderate deviation principles for alternating Engel expansions, \emph{Journal of Number Theory} 156 (2015), 263--276.



\bibitem{Falconer1997} {K.~Falconer,} \textsl{Techniques in Fractal Geometry},
    John Willey and Sons, Ltd. (Chichester,  1997).

\bibitem{Falconer2004} { K.~Falconer,} \textsl{Fractal Geometry: Mathematical Foundations and Applications}, Wiley, 2004.

\bibitem{HW1979} 
 G.~H.~Hardy  and   E.~M.~Wright, \emph{An Introduction to the Theory of Numbers}, 5th ed., Oxford University
Press  (Oxford, 1979).

\bibitem{HRW2000} S. Hua,  H. Rao, Z. Wen et al., On the structures and dimensions of Moran sets, \emph{Sci. China Ser. A-Math.} \textbf{43} (2000), No. 8,  836--852 doi:10.1007/BF02884183

\bibitem{HLW2020-JMAA}
Liangyi Huang, Qinghui Liu, and GuizhenWang, Multifractal analysis of Bernoulli measures on a class of homogeneous Cantor sets, \emph{J. Math. Anal. Appl.} 491 (2020), 124362.



\bibitem{IS2009}
S.~Ito and T. ~Sadahiro,  Beta-expansions with negative bases, \emph{Integers},  \textbf{9}  (2009),   
   239--259.

    

\bibitem{10}	S. Kalpazidou, A. Knopfmacher, J. Knopfmacher. L\"uroth-type alternating series representations for real numbers, \emph{Acta Arithmetica}  55 (1990),  311--322. 


\bibitem{Kawamura2010}
K.~Kawamura, The derivative of Lebesgue's singular function, \emph{ Real Analysis Exchange}, 
Summer Symposium 2010,  83--85.

\bibitem{Kruppel2009}
M.~Kruppel, De Rham's singular function, its partial derivatives with
respect to the parameter and binary digital sums, \emph{Rostock. Math. Kolloq.},
\textbf{64} (2009), 57--74.

\bibitem{15}	J. L\"uroth. Ueber eine eindeutige Entwickelung von Zahlen in eine unendliche Reihe, \emph{Math. Ann.} 21 (1883), 411--423. 

\bibitem{Mandelbrot1977} {B.~Mandelbrot, } \textsl{Fractals: Form, Chance and Dimension}, 
       W\!.\,H. Freeman and Co.  (San Francisco, Calif., 1977).

\bibitem{Mandelbrot1999} {B.~Mandelbrot, } \textsl{The Fractal Geometry of Nature}, 
    18th printing,  Freeman (New York, 1999).

\bibitem{Moran1946} {P.~A.~P.~Moran, } {Additive functions of intervals and Hausdorff measure}, \emph{ Proc. Cambridge Philos. Soc.},  \textbf{42} (1946), no.\,1, 15--23,
 doi:10.1017/S0305004100022684.




\bibitem{Minkowski}
H.~Minkowski,  Zur Geometrie der Zahlen. In: Minkowski, H. (ed.) Gesammeine Abhandlungen, Band 2,
 50--51, Druck und Verlag von B. G. Teubner  (Leipzig und Berlin, 1911).


\bibitem{OSS1995}
T.~Okada, T.~Sekiguchi, and Y.~Shiota, An explicit formula of the
exponential sums of digital sums, \emph{Japan~J. Indust. Appl. Math.}, \textbf{ 12} (1995),
425--438.

\bibitem{Olsen1995}
L. Olsen, A multifractal formalism,  \emph{Adv. Math.} 116 (1995), 82---196. 

\bibitem{PW1996}  Yakov Pesin and Howard Weiss, On the Dimension of Deterministic and Random Cantor-like Sets, Symbolic Dynamics, and the Eckmann-Ruelle Conjecture, \emph{Commun. Math. Phys} \textbf{182} (1996),  105--153 doi:10.1007/BF02506387 

\bibitem{Engel 3}
T. A.~Pierce, On an algorithm and its use in approximating roots of algebraic equations,
\emph{Amer. Math. Monthly} 36 (1929), 523--525.


\bibitem{PS1995}  M. Pollicott and K\'aroly Simon, The Hausdorff dimension of $\lambda$-expansions with deleted digits, 
 \emph{Trans. Amer. Math. Soc.}  \textbf{347} (1995), 967-983 https://doi.org/10.1090/S0002-9947-1995-1290729-0 

\bibitem{RS2021-JGA}
Mrinal Kanti Roychowdhury and Bilel Selmi, Local Dimensions and Quantization Dimensions in Dynamical Systems, \emph{J. Geom. Anal.} 31 (2021), 6387-6409.


\bibitem{PVB2011}  J.~Parad\'is, P.~Viader, and Ll.~Bibiloni,  A New Singular Function, \emph{ Amer. Math. Monthly}, \textbf{ 118}, no. 4, 344-354 (2011) , DOI: 10.4169/amer.math.monthly.118.04.344


\bibitem{Renyi1957}
A.~R\'enyi, Representations for real numbers and their ergodic properties, \emph{ Acta. Math. Acad. Sci. Hungar.}, \textbf{8} (1957),  477--493.



\bibitem{Salem1943}{ R.~Salem},
 {On some singular monotonic functions which are stricly increasing,}
 {\itshape Trans. Amer. Math. Soc.},  {\bf 53} (1943), 423--439.


\bibitem{Selmi2022-ASM}
B. Selmi, The relative multifractal analysis, review and examples, \emph{Acta Scientiarum Mathematicarum}
86 (2020), 635-666.

\bibitem{Selmi2022-ASM}
B. Selmi,  A review on multifractal analysis of Hewitt-Stromberg measures \emph{ J. Geom. Anal.} 32(1)
(2022), 1-44.

\bibitem{Selmi2021-BPASciM}
B. Selmi, The mutual singularity of multifractal measures for some non-regularity Moran fractals, 
\emph{Bulletin Polish Acad. Sci. Math.} 69 (2021), 21-35.



\bibitem{S. Serbenyuk abstract 6}{S.~O.~Serbenyuk},
 {On one nearly everywhere continuous and  nowhere differentiable function defined by automaton with finite memory,}
 {\itshape International Scientific Conference ``Asymptotic Methods in the Theory of Differential Equations" dedicated to 80th anniversary of M.~I.~Shkil:}  Abstracts, Kyiv: National Pedagogical Dragomanov University, 2012.  P. 93 (Ukrainian), available at  https://www.researchgate.net/publication/311665377

\bibitem{Symon12(2)}{S.~O.~Serbenyuk},
 {On one nearly everywhere
continuous and nowhere differentiable function, that defined by automaton with 
finite memory,}
 {\itshape Naukovyi Chasopys NPU im. M. P. Dragomanova. Ser. 1. Phizyko-matematychni Nauky  [Trans. Natl. Pedagog. Mykhailo Dragomanov Univ. Ser. 1.
Phys. Math.]}, {\bf 13(2)} (2012). (Ukrainian),  available at https://www.researchgate.net/publication/292970012

 \bibitem{S. Serbenyuk abstract 7}{S.~O.~Serbenyuk},
 {On one generalization of functions defined by automatons with finite memory,}
{\itshape Third Interuniversity Scientific Conference of Young Scientists on Mathematics and Physics:}   Abstracts, Kyiv: National University of Kyiv-Mohyla Academy, 2013. P. 112--113  (Ukrainian), available at https://www.researchgate.net/publication/311414454


\bibitem{S. Serbenyuk abstract 8}{\itshape S.  Serbenyuk},
 { On two functions with complicated local structure,}
 {\itshape Fifth International Conference on Analytic Number Theory and Spatial Tessellations:}   Abstracts, Kyiv: Institute of Mathematics of the National Academy of Sciences of Ukraine and Institute of Physics and Mathematics of the National Pedagogical Dragomanov University, 2013. P. 51--52, available at https://www.researchgate.net/publication/311414256



\bibitem{Symon2015}{S.~O.~Serbenyuk},
 {Functions, that defined by functional equations systems in terms of Cantor series representation of numbers,}
 {\itshape Naukovi Zapysky NaUKMA}, {\bf 165} (2015), 34--40. (Ukrainian),  available at https://www.researchgate.net/publication/292606546

\bibitem{Symon2017}{ S.~O.~Serbenyuk},
 {Continuous Functions with Complicated Local
Structure Defined in Terms of Alternating Cantor
Series Representation of Numbers,}
 {\itshape Journal of Mathematical Physics, Analysis, Geometry (Zh.~Mat.~Fiz.~Anal.~Geom.)}, {\bf 13} (2017), No. 1,  57--81. 
https://doi.org/10.15407/mag13.01.057



\bibitem{Serbenyuk-2016}
S.~Serbenyuk, On one class of functions with complicated local structure, \emph{\v{S}iauliai
Mathematical Seminar}, \textbf {11 (19)} (2016), 75--88.


 

\bibitem{S.Serbenyuk 2017} S.~Serbenyuk, On one fractal property of the Minkowski function, \emph{Revista de la Real Academia de Ciencias Exactas, F\' isicas y Naturales. Serie A. Matem\' aticas}, \textbf{112} (2018), no.~2, 555--559, doi:10.1007/s13398-017-0396-5



 
\bibitem{S. Serbenyuk systemy rivnyan 2-2}{ S.~O.~Serbenyuk},
Non-Differentiable functions defined in~terms of~classical representations of~real numbers,
 \textit{Journal of Mathematical Physics, Analysis, Geometry  (Zh.~Mat.~Fiz.~Anal.~Geom.)}, \textbf{14} (2018), no.~2,
 197--213. https://doi.org/10.15407/mag14.02.197



\bibitem{S.Serbenyuk} \emph{S.~Serbenyuk } On some generalizations of real numbers representations, arXiv:1602.07929v1 (in Ukrainian)

\bibitem{preprint1-2018}
S.~Serbenyuk, Generalizations of certain representations of real numbers, \emph{Tatra Mountains Mathematical Publications}, \textbf{77} (2020), 59--72, https://doi.org/10.2478/tmmp-2020-0033, 
arXiv:1801.10540.

\bibitem{Symon2019} S.~Serbenyuk, On one application of infinite systems of functional equations in function theory, \emph{ Tatra Mountains Mathematical Publications}, \textbf{74} (2019), 117-144, https://doi.org/10.2478/tmmp-2019-0024  

\bibitem{preprint2019} S.~Serbenyuk, Generalized shift operator of certain encodings of real numbers, arXiv:1911.12140v1, 6 pp. 

\bibitem{preprint19} S.~Serbenyuk, On certain functions and related problems,  arXiv:1909.03163 


\bibitem{S2023-Communications-in-Mathematics} Symon Serbenyuk. (2023) Cantor series expansions of rational numbers. \emph{Communications in Mathematics} {\bf{31}}, No. 1. https://doi.org/10.46298/cm.10454

\bibitem{Symon2021} S.~Serbenyuk, Systems of functional equations and generalizations of certain functions, \emph{Aequationes  Mathematicae},   \textbf{95} (2021), 801--820, https://doi.org/10.1007/s00010-021-00840-8 

\bibitem{Symon2023} S.~Serbenyuk, Functional equations, alternating expansions, and generalizations of the Salem functions,  \emph{Aequationes Mathematicae} (2023), https://doi.org/10.1007/s00010-023-00992-9

\bibitem{Symon21} S.~Serbenyuk, Certain functions defined in terms of Cantor series, \emph{Journal of Mathematical Physics, Analysis, Geometry (Zh.~Mat.~Fiz.~Anal.~Geom.)}, { \bf{16}} (2020), no. 2, 174--189,  https://doi.org/10.15407/mag16.02.174 

\bibitem{Symon21-1} S.~Serbenyuk, On certain maps defined by infinite sums, \emph{ The Journal of Analysis}, \textbf{ 28} (2020), 987--1007. https://doi.org/10.1007/s41478-020-00229-x

\bibitem{Symon2023} S.~Serbenyuk,  A certain modification of classical singular function, \emph{ Bolet\'in de la Sociedad Matem\'atica Mexicana},  {\bf{29}} (2023), no.3, Article Number 88, https://doi.org/10.1007/s40590-023-00569-1

\bibitem{Symon2023-numeral-systems}  Symon Serbenyuk. Some types of numeral systems and their modeling, {The Journal of Analysis } { {31}} (2023), 149--177. https://doi.org/10.1007/s41478-022-00436-8

\bibitem{2016Serbeyuk-KNU}
Symon Serbenyuk, Nega-$\tilde Q$-representation as a generalization of certain alternating representations of real numbers, \emph{Bulletin of the Taras Shevchenko National University of Kyiv Mathematics and Mechanics} 1 (35) (2016), 32--39. Ukrainian. Available at https://www.researchgate.net/publication/308273000.

 
\bibitem{sets1} S.~Serbenyuk,  Some Fractal Properties of Sets Having the Moran Structure, \emph{Tatra Mountains Mathematical Publications}, \textbf{81}, no.1, 3922,  1--38. https://doi.org/10.2478/tmmp-2022-0001

\bibitem{sets2}  S.~Serbenyuk,  Certain Singular Distributions and Fractals, \emph{Tatra Mountains Mathematical Publications},  \textbf{79}, no.2, 3921, 163--198. https://doi.org/10.2478/tmmp-2021-0026

\bibitem{sets} S.~O.~Serbenyuk,  One distribution function on the Moran sets. \emph{ Azerb.~J.~Math.}, \textbf{10} (2020), no.2, 12--30, arXiv:1808.00395.



\bibitem{Serbenyuk20Tatra}  S.~Serbenyuk, A note on expansions of rational numbers by certain series, {\it Tatra Mountains Mathematical Publications}, {\bf 77} (2020), 53-58, DOI: 10.2478\textfractionsolidus tmmp-2020-0032,  arXiv:1904.07264

\bibitem{Symon2021-ActaMathHungarica}
Symon Serbenyuk, Rational numbers defined in terms of certain generalized series, \emph{Acta Mathematica  Hungarica}  {\bf 164} (2021), 580--592. https://doi.org/10.1007/s10474-021-01163-5

\bibitem{Symon2024} Serbenyuk, Symon. Singular Modifications Of A Classical Function. \emph{ Acta Mathematica  Hungarica}  (2024). https://doi.org/10.1007/s10474-024-01406-1

\bibitem{Sumi2009}
H.~Sumi, Rational semigroups, random complex dynamics and singular
functions on the complex plane, \emph{ Sugaku},  \textbf{61} (2009), no.~2,  133--161.

\bibitem{Takayasu1984}
H.~Takayasu, Physical models of fractal functions, \emph{Japan~J.~Appl.~Math.}
\textbf{1} (1984), 201--205.

\bibitem{TAS1993}
S.~Tasaki, I.~Antoniou, and Z.~Suchanecki, Deterministic diffusion,
De Rham equation and fractal eigenvectors, \emph{Physics Letter~A}, \textbf{179} 
(1993), no.~1,  97--102.


\bibitem{Wu2005}
 M. Wu, The multifractal spectrum of some Moran measures, \emph{Sci. China. Ser. A Math.}, {\bf{48}} (2005), 97--112.

\bibitem{Wu2005(2)}
M. Wu, The singularity spectrum $f(\alpha)$ of some Moran fractals, \emph{Monatsh Math.}, {\bf{144}} (2005), 141--155.

\bibitem{WX2011-ChSF}
M. Wu and J. Xiao, The singularity spectrum of some non-regularity Moran fractals, \emph{Chaos, Solitons \& Fractals} 44 (2011), 548-557.

\bibitem{XW2008}
J. Xiao and M.Wu, The multifractal dimension functions of homogeneous Moran measure, \emph{Fractals}, {\bf{16}} (2008), 175--185.

\bibitem{XW2008-Fr}
J. Xiao and M.Wu, The multifractal dimension functions of homogeneous Moran measure, \emph{Fractals}
16 (2008), 175-185.

\bibitem{Yuan2019-N}
Z. Yuan, Multifractal spectra of Moran measures without local dimension, \emph{Nonlinearity} 32 (2019),
5060-5086.



\bibitem{Zamfirescu1981} T.~Zamfirescu, Most monotone functions are singular, \emph{Amer. Math. Mon.}, \textbf{88} (1981), 47--49.

\bibitem{Wikipedia-pathology} Wikipedia Contributors, Pathological (mathematics), Wikipedia, the free encyclopedia, available as https://en.wikipedia.org/wiki/Pathological\_(mathematics) (July 24, 2023)
  

\end{thebibliography}
\end{document}